\documentclass[final,leqno,letterpaper]{etna}
\setbibdata{1}{Last}{xx}{202x} 

\usepackage{algorithm}
\usepackage{algorithmic}
\usepackage{amsfonts}
\usepackage{amsmath}
\usepackage{amssymb}
\usepackage{enumitem}
\usepackage{graphicx}
\usepackage{lipsum}
\usepackage{hyperref}
\usepackage{mathtools}
\usepackage{microtype}
\usepackage{color}

\newcommand{\x}{\mathbf{x}}
\newcommand{\y}{\mathbf{y}}

\hypersetup{%
    pdftitle={The structured distance to singularity of a symmetric tridiagonal Toeplitz matrix},
    pdfauthor={Silvia Noschese},
    pdfkeywords={a keyword, another keyword, and another one}
}

\title{The structured distance to singularity of a symmetric tridiagonal Toeplitz  matrix\thanks{%
Received... Accepted... Published online on... Recommended by....
Research partially supported by a grant from SAPIENZA 
Universit\`a di Roma and by INdAM-GNCS.
}}

\author{Silvia Noschese\footnotemark[2]}

\shorttitle{The structured distance to singularity of a symmetric tridiagonal Toeplitz matrix} 
\shortauthor{S.~Noschese}

\begin{document}

\maketitle

\renewcommand{\thefootnote}{\fnsymbol{footnote}}

\footnotetext[2]{Dipartimento di Matematica ``Guido Castelnuovo'', SAPIENZA 
Universit\`a di Roma, P.le A. Moro, 2, I-00185 Roma, Italy. E-mail: 
{\tt noschese@mat.uniroma1.it}.}

\begin{center}
{\it Dedicated to Lothar Reichel on the occasion of his 70th birthday.}
\end{center}

\begin{abstract}
This paper is concerned  with the distance of a symmetric tridiagonal Toeplitz matrix $T$ to the  manifold of similarly structured singular matrices, and with determining the closest matrix to $T$ in this manifold.  Explicit formulas are presented, that exploit the analysis of the sensitivity of the spectrum of $T$ with respect to structure-preserving perturbations of its entries. 

\end{abstract}



\section{Introduction} 
The sensitivity of the solution of a linear algebra problem with respect to perturbations in the data has received considerable attention in the literature.
 This interest can be traced back to the work of Golub and Wilkinson \cite{GW76}, Wilkinson \cite{Wi84, Wi}, Demmel \cite{De87,De92}, and numerous other authors; see, e.g., \cite{GV,SS} and the references therein. The effect of structure-preserving perturbations of the matrix entries has been analyzed in the context of linear systems (see, e.g., \cite{Va90,HH,R12003,R22003}) as well as  for structured eigenproblems.  Measures of structured eigenvalue sensitivity include both structured pseudospectra (see, e.g., \cite{Gr,R2006,D2009,Ka,BGN,NR17}) and {\it structured condition numbers}, introduced by Tisseur \cite{T2003}.  The latter measures have been investigated, e.g., in \cite{KKT,NP06,NPR} and, recently, applied to network analysis \cite{NR22, ENR}.
 
Symmetric tridiagonal Toeplitz matrices arise in several applications, including the numerical solution of ordinary and partial differential equations \cite{Sm,BG05}, and as regularization matrices in Tikhonov regularization for the solution of discrete ill-posed problems \cite{Ha, RY, HNR}. It is therefore important to understand properties of such  structured matrices relevant for computation.

By leveraging structured eigensensitivity analysis, this paper concerns with the distance of symmetric tridiagonal Toeplitz matrices to the  manifold of similarly structured singular matrices, and with determining the closest matrix in this manifold.   The structured  measures we deal with can be computed by endowing the considered subspace of matrices with the Frobenius norm. We remark that estimates on the distance to singularity in the spectral norm of symmetric tridiagonal Toeplitz matrices, with respect to normwise structured perturbations, can be found in the work of Rump \cite{R12003}, however the present paper develops a novel approach based on the structured eigensensitivity analysis that gives rise to an explicit formula for the {\it structured distance to singularity} in the Frobenius norm of such matrices and to upper bounds for their structured distance to singularity in the spectral norm, except for the definite case where one has an explicit formula  also for the latter distance.

We denote symmetric tridiagonal Toeplitz matrices in ${\mathbb R}^{n\times n}$ by
\begin{equation}\label{TRap}
T=(n;\delta, \sigma)=\left[ 
\begin{array}{ccccccc}
\delta & \sigma &  &  &  &  & \Large{O} \\ 
\sigma & \delta & \sigma &  &  &  &  \\ 
& \sigma & \cdot & \cdot &  &  &  \\ 
&  & \cdot & \cdot & \cdot &  &  \\ 
&  &  & \cdot & \cdot & \cdot &  \\ 
&  &  &  & \cdot & \cdot & \sigma \\ 
\Large{O} &  &  &  &  & \sigma & \delta
\end{array}
\right].
\end{equation}
It is well known that the eigenvalues of (\ref{TRap})  are given by 
\begin{equation}\label{lamh}
\lambda_{h}=\delta +2\sigma\,\cos \frac{h\pi}{n+1}, \quad h=1, \dots,n;
\end{equation}
see, e.g., \cite{Sm}. 
 Throughout this paper we refer to $\lambda_{h}$ in \eqref{lamh}  as the $h$th eigenvalue of $T=(n;\delta, \sigma)$. Assume $\sigma \ne 0$. Then $T$ has $n$ simple real eigenvalues
allocated symmetrically with respect 
to $\delta$.  The components of the eigenvector 
$\x_{h}=[x_{h,1},x_{h,2},\ldots,x_{h,n}]^T$ 
associated with the eigenvalue
$\lambda _{h}$ and normalized 
to have unit Euclidean norm, are the following,
\begin{equation}\label{xhk}
x_{h,k}=\sqrt{\frac{2}{n+1}}\sin \frac{hk\pi}{n+1},\quad k=1, \dots,n,\quad h=1, \dots,n;
\end{equation}
see, e.g., \cite{NPR}. Let $\Lambda$ and $X$ denote the diagonal matrix with $\lambda_h$ as $h$th diagonal entry
and the orthogonal matrix with $\x_{h}$ as $h$th column, respectively.

Let $\mathcal{S}$ denote the manifold of singular matrices in  $\mathbb{R}^{n\times n}$.
The distance of $T $ to $\mathcal{S}$ in the Frobenius norm  is given by
\[
d_{F}(T)=\min_{S\in 
{\mathcal S}} \left\| T-S\right\|_{F},
\]
where $\|\cdot\|_F$ stands for the Frobenius norm.
As it is well known, if $\lambda _{k}$ is an eigenvalue of $T$ such that $|\lambda_k|
<|\lambda_h|$, with $h=1,\dots, n, \, h\ne k$, then, thanks to the Eckart-Young theorem \cite{EY}, one has that
\begin{equation}\label{Sstar}
S_*:=\sum_{h=1,\dots, n, h\ne k}\lambda _{h}\x_h\x_h^T
\end{equation}
is the unique matrix in ${\mathcal S}$ such that
\begin{equation}\label{dFT}
d_{F}(T)= \left\| T- S_* \right\|_{F}
=\left\| \lambda_k\x_k\x_k^T\right\|_{F}=|\lambda _{k}|.
\end{equation}
Throughout this paper the superscript  $(\cdot)^T$ stands for transposition and $(\cdot)^H$
for transposition and complex conjugation.  We explicitly observe that Equation \eqref{dFT} holds true if $\lambda _{k}$ is such that $|\lambda_k|
\leq |\lambda_h|$, for $h\ne k$. However, if not all the above inequalities are strict, then $S_*$ in \eqref{Sstar} is not the unique matrix that attains the minimum $\min_{S\in 
{\mathcal S}} \left\| T-S\right\|_{F}$ .

Let ${\cal T}$ denote the subspace of $\mathbb{R}^{n\times n}$ formed by 
symmetric tridiagonal Toeplitz matrices. The structured distance to singularity
of $T$ in the Frobenius norm is given by
\[
d_{F}^{\cal T}(T)=\min_{S\in 
{\mathcal S}\cap{\mathcal T}} \left\| T-S\right\|_{F}.
\]

\noindent Notice that the singular matrix closest  to $T$, i.e., the matrix $S_*$ defined in \eqref{Sstar}, 
is not a tridiagonal Toeplitz matrix, that is to say, $S_* \notin {\cal T}$. Hence, by \eqref{dFT}, one has
\[
d_{F}^{\cal T}(T)> \min_{h=1,\dots,n}|\lambda _{h}|\,.
\]

The property of any $T=(n;\delta, \sigma)$ of  fixed order $n$, with $\sigma\ne0$, to have the same eigenvectors \eqref{xhk}, 
as well as  the knowledge of the worst-case
structured perturbation  for  the eigenvalue $\lambda _{k}$ (see, e.g., \cite{NR19})
are the tools used  in this paper to shed light on the structured distance to singularity of a
symmetric tridiagonal Toeplitz matrix. 

For any matrix $A\in{\mathbb R}^{n\times n}$ let $A|{\cal T}$ denote the matrix in the subspace 
 ${\cal T}$ closest to $A$ with respect to the Frobenius norm. It is straightforward to verify that, in order to have such projection of $A$  in ${\cal T}$,  one takes the symmetric part of the tridiagonal Toeplitz matrix obtained in turn  by replacing in each structure diagonal all the entries of  $A$ with their arithmetic mean. 
 Notice that  the projection $S_*|{\cal T}\notin {\mathcal S}\cap{\mathcal T}$.  In the present paper, we give an explicit formula for the structured distance to singularity of $T\in{\cal T}$, based on 
 the ratios
\[
\frac{|\lambda _{h}|}{\left\| (\x_h\x_h^T)|{\cal T}\right\|_{F}}\,, \qquad h=1,\dots,n,
\] 
 and we illustrate how to find the  closest 
 matrix $S_*^{\cal T}\in {\mathcal S}\cap{\mathcal T}$ to $T$.  
 Additionally, if $T\in{\cal T}$ is positive or negative definite, we show that 
\[
\frac{d_{F}(T)}{d_{F}^{\cal T}(T)}\approx\sqrt{\frac{3}{n}}\,.
\]
 
This paper is organized as follows. Section \ref{sec2} is concerned with the sensitivity analysis of
the spectrum of a symmetric  tridiagonal Toeplitz matrix $T$ with respect to structure-preserving perturbations and
Section \ref{sec22} exploits such analysis to discuss an upper bound for the structured distance to singularity of $T$ in the Frobenius norm.
Section \ref{sec3}  is concerned with determining $S_*^{\cal T}$, i.e., the unique closest singular symmetric  tridiagonal Toeplitz matrix to $T$
in the Frobenius norm, as well as its distance $d_{F}^{\cal T}(T)$ to  $T$. Section \ref{sec4} discusses the structured 
distance to singularity of a symmetric positive definite tridiagonal Toeplitz matrix both in the Frobenius norm and in the spectral norm and shows monotonicity properties of the entries of its Cholesky factor. In Section \ref{secx} the cases relevant to indefinite matrices in ${\cal T}$ are investigated.
Numerical examples are presented in  
Section \ref{sec5} and concluding remarks can be found in Section \ref{sec6}.

\section{Eigenvalue structured sensitivity}\label{sec2}
Let $\mathcal{C}$ be any subspace of ${\mathbb C}^{n\times n}$ formed by matrices
with a given symmetry-structure, that is to say a structure that exhibits a kind of symmetry, like reflection or translation. Examples of structured matrices of such kind are complex symmetric, skew-symmetric, persymmetric, skew-persymmetric matrices, or, more specifically, Toeplitz and Hankel  matrices; see, e.g.,  \cite[Section 5]{NP07}. 
Let $\lambda\in{\mathbb C}$ be a simple eigenvalue of a given matrix $C\in\mathcal{C}$, 
with corresponding  right and left eigenvectors
$\x$ and $\y$ of unit Euclidean norm. 
A well-known indicator of the
sensitivity of $\lambda$ to perturbations  of the entries of $C$ is given by analyzing the 
coefficient of the first term  in the expansion of powers of $\varepsilon$ of $\lambda(\varepsilon)$, 
where $\lambda(\varepsilon)$ is the eigenvalue of $C+\varepsilon E$, with $\|E\|_F=1$, that tends toward $\lambda$ when
$\varepsilon$ goes to zero:
\[
\lambda(\varepsilon)=\lambda+\frac{\y^HE\x}{\y^H\x}\,\varepsilon +O(\varepsilon ^2).
\]
In fact, the traditional condition number $\kappa(\lambda)$  is the first-order measure of the worst-case effect on $\lambda$ of perturbations in $C$, i.e., 
\[
\kappa(\lambda)=\max_{E\in {\mathbb C}^{n\times n},\,\, \|E\|_F=1}\left|\frac{\y^HE\x}{\y^H\x}\right|=\frac{1}{\y^H\x},
\]
and the {\it worst-case perturbation} is proven to be $E=\y\x^H$; see \cite[Section\,2]{Wi}. 

We are interested in  the sensitivity of the eigenvalues of $C\in{\cal C}$ with respect to structure-preserving perturbations.  Let $(\y\x^H)|{\cal C}$  denote the matrix in the subspace ${\cal C}$ closest to $\y\x^H$ with respect to the Frobenius norm. Following the approach above, and admitting only unit-norm perturbation matrices $E$ that belong to $\mathcal{C}$,
leads to the {\it structured condition number}  $\kappa^{\cal C}(\lambda)$ as the first-order measure of the worst-case effect on $\lambda$ of perturbations of the same structure as $C$, that is to say,
\[
\kappa^{\cal C}(\lambda)=\max_{E\in {\cal C},\,\, \|E\|_F=1}\left|\frac{\y^HE\x}{\y^H\x}\right|=\frac{\|(\y\x^H)|{\cal C}\|_{F}}{\y^H\x};
\]
see \cite{KKT,NP07}. Thus
\[
\frac{\kappa^{\cal C}(\lambda)}{\kappa(\lambda)}=\|(\y\x^H)|{\cal C}\|_{F}\leq 1.
\]
Moreover, thanks to \cite[Lemma 3.2]{NP07}, which claims that
\begin{equation}\label{nrmsq}
\y^H(\y\x^H) |{\cal C}\,\x=\|(\y\x^H) |{\cal C}\|^2_{F},
\end{equation}
one has that the {\it worst-case structured perturbation} is given by 
\[
E=\frac{\y\x^H |{\cal C}}{\|\y\x^H |{\cal C}\|_{F}}.
\]

Consider  now the  symmetry-structured subspace ${\cal T}$ of $\mathbb{R}^{n\times n}$. Since $T \in {\cal T}$ is 
symmetric, then for any (real) eigenvalue $\lambda_h$ in \eqref{lamh}, one has 
$\x_h=\y_h\in {\mathbb R}^n$ and 
\begin{equation}\label{klam0}
\kappa^{\cal T}(\lambda_h)=\|(\x_h\x_h^T)|{\cal T}\|_{F}.
\end{equation}
Explicit formulas for the  structured condition numbers $\kappa^{\cal T}(\lambda_h)$ for 
the eigenvalues  $\lambda_h$, with $h=1,\dots, n$, of $T=(n;\delta,\sigma)$, with $\sigma\ne 0$, 
are the following,
\begin{equation}\label{klam}
\kappa^{\cal T}(\lambda_h)=\sqrt{\frac{1}{n} +\frac{2}{n-1}\cos^2 \frac{h\pi}{n+1}},\;\;\;\;h=1,\dots, n;
\end{equation}
see \cite[Proposition 4.11]{NR19}.  As the eigenvectors \eqref{xhk}, the  structured
condition numbers \eqref{klam} do not depend on $\delta$ and $\sigma$.
Figure \ref{f_klam} shows the structured condition numbers of the eigenvalues of  
a $100\times 100$ symmetric tridiagonal Toeplitz matrix;  see \cite[Figure 2]{NR19}. 
\begin{figure}[tbp]

\centering
\includegraphics[scale=0.45]{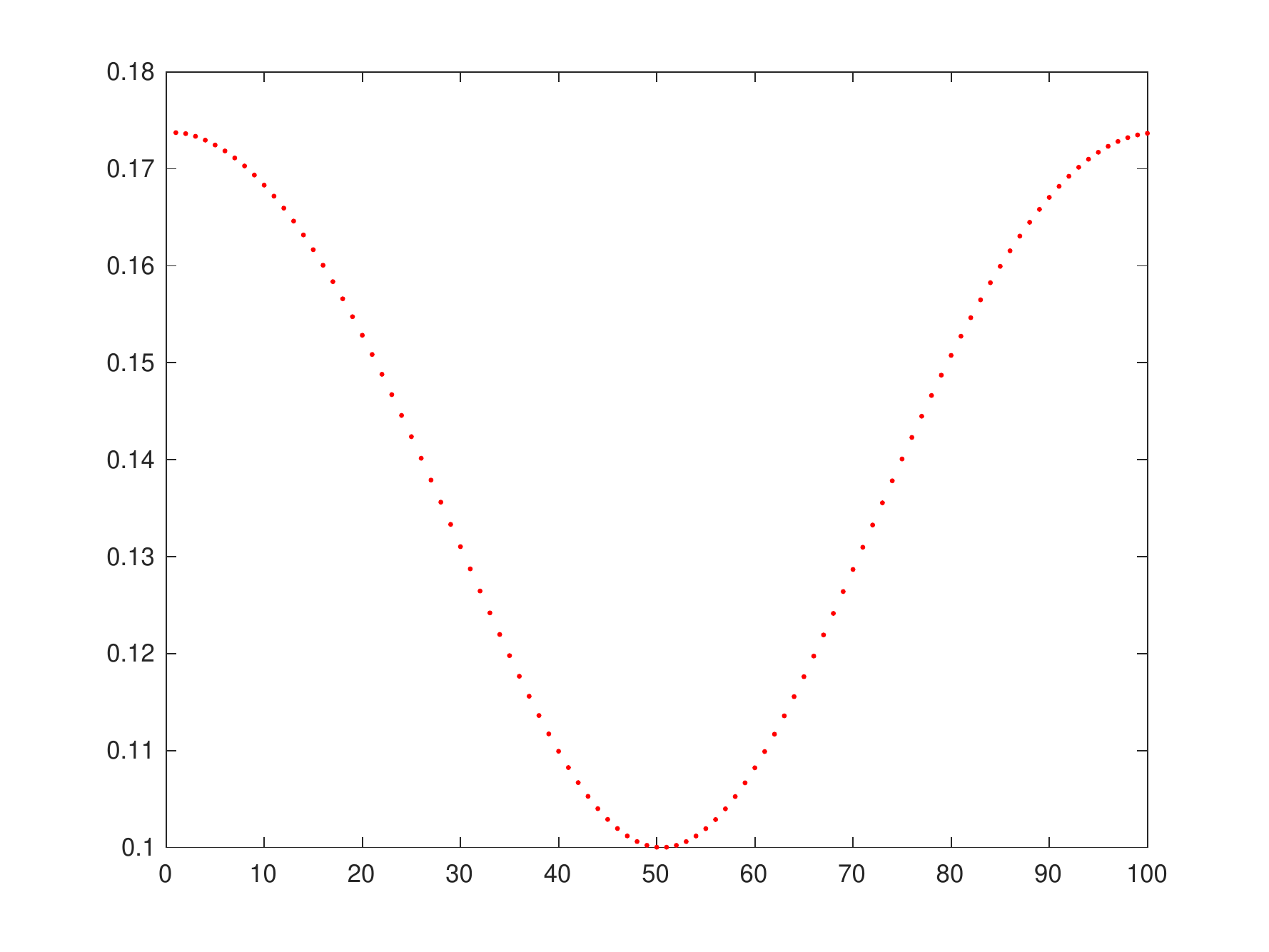}
\caption{Structured eigenvalue condition numbers for the matrix 
$T=(100;\delta,\sigma)$, where  $\sigma$ and $\delta$ are arbitrarily chosen real parameters. 
The horizontal axis shows the index of the eigenvalue $\lambda_h$, $h=1,\dots,100$, and the vertical axis 
the structured condition numbers $\kappa^{\cal T}(\lambda_h)$ in \eqref{klam}.}\label{f_klam}
\end{figure}

The minimum value in \eqref{klam} is attained at $h=(n+1)/2$, if $n$ is odd:
\begin{equation*}
\kappa^{\cal T}(\lambda_{(n+1)/2})=\sqrt{\frac{1}{n}}.
\end{equation*}
If $n$ is even, the smallest structured condition numbers are relevant to the two most central eigenvalues in the spectrum of $T$, having
\[
\kappa^{\cal T}(\lambda_{n/2})=\kappa^{\cal T}(\lambda_{(n+2)/2})=\sqrt{\frac{1}{n} +\frac{2}{n-1}\cos^2 \frac{n\pi}{2(n+1)}}\geq\sqrt{\frac{1}{n}}\,,
\]
and,  for large dimension $n$,
\begin{equation}\label{n2}
\kappa^{\cal T}(\lambda_{n/2})=\kappa^{\cal T}(\lambda_{(n+2)/2})\approx\sqrt{\frac{1}{n}}.
\end{equation}
Conversely,  the two extremal eigenvalues  in the spectrum of $T$ have  the largest structured condition numbers, and for large $n$ one has the following estimate
\begin{equation}\label{n3}
\kappa^{\cal T}(\lambda_1)=\kappa^{\cal T}(\lambda_n)\approx\sqrt{\frac{1}{n} +\frac{2}{n-1}}
=\sqrt{\frac{3n-1 }{n^2-n}}\approx\sqrt{\frac{3}{n}}\,.
\end{equation}
Figure \ref{E} provides an illustration of the unit norm matrices $\x_1\x_1^T$ and $\x_{50} \x_{50}^T$ for a $100\times 100$ symmetric tridiagonal Toeplitz matrix.  It is apparent that the weight of the three central diagonals of the worst-case perturbation $\x_1\x_1^T$ for the extremal eigenvalue $\lambda_1$, shown in the left picture,  is larger than the weight of the relevant diagonals of the worst-case perturbation $\x_{50} \x_{50}^T$ for the  central eigenvalue $\lambda_{50}$, depicted in the right picture. 
In fact, for $n\geq 10$, the arithmetic means of the three central diagonals of the worst-case perturbation matrices for the two  extremal eigenvalues  of  $T$ give rise to structured condition numbers 
more than $70\%$ larger than the ones produced by the arithmetic means of the relevant diagonals of the worst-case perturbation matrices for the two most central eigenvalues.
Figure \ref{ratios} shows the ratios $\kappa^{\cal T}(\lambda_{n})/\kappa^{\cal T}(\lambda_{n/2})$, for even $n$, and $\kappa^{\cal T}(\lambda_{n})/\kappa^{\cal T}(\lambda_{(n+1)/2})$, for odd $n$, with $n=2,\dots,100$.

\begin{figure}[tbp]
\centering
\includegraphics[scale=0.34]{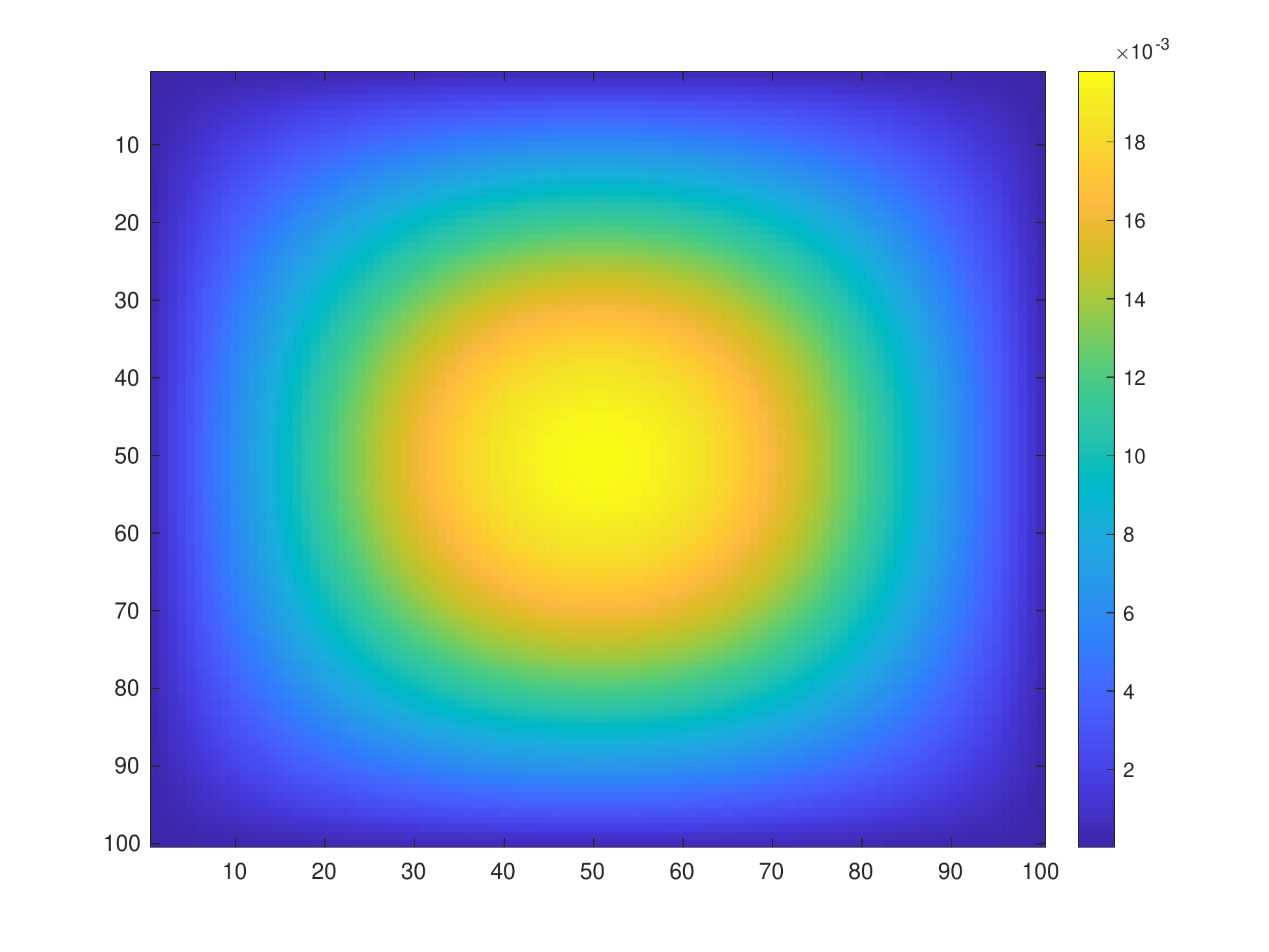}\includegraphics[scale=0.34]{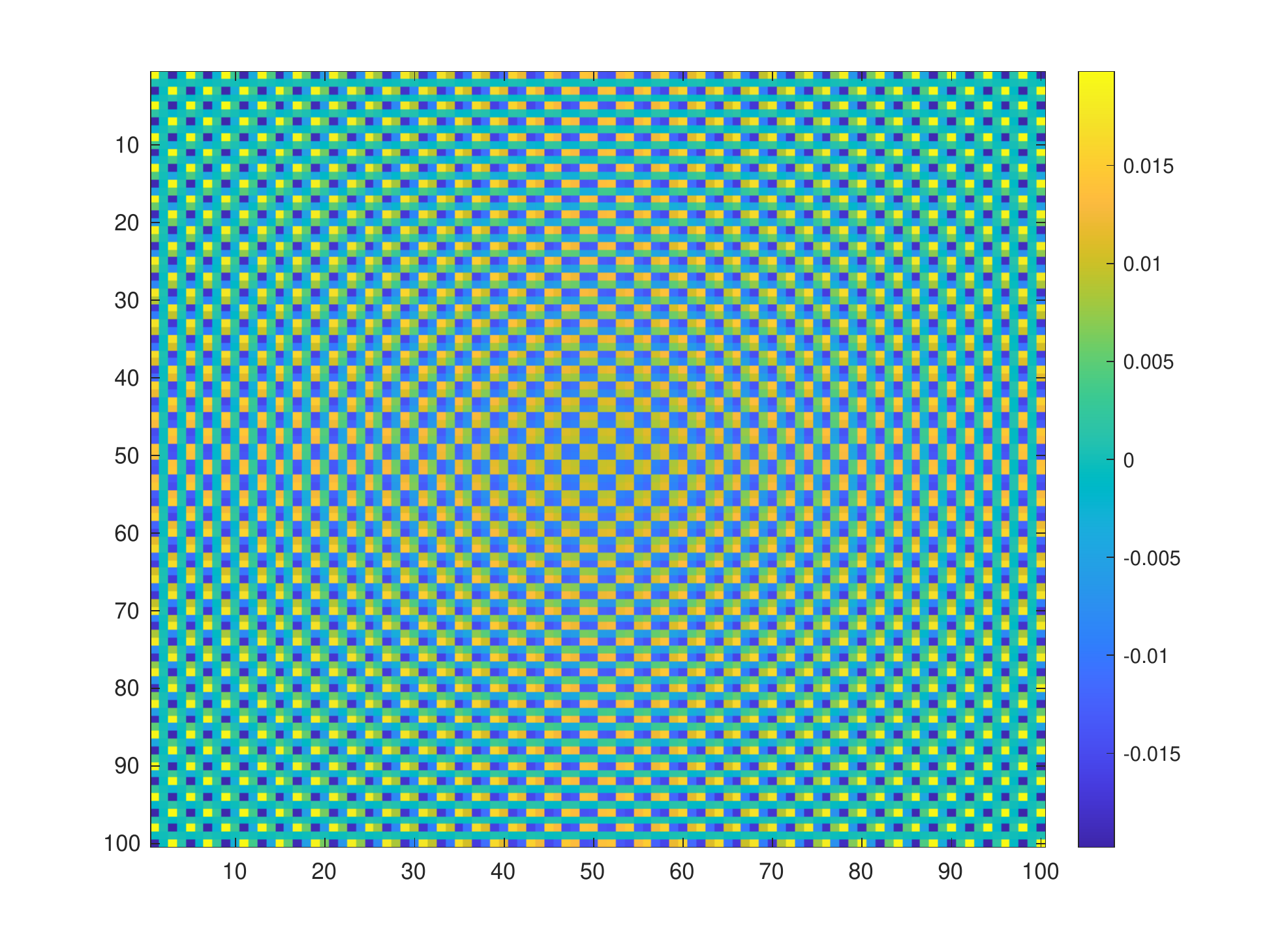}
\caption{Images of  the worst-case perturbations $\x_1\x_1^T$ and $\x_{50} \x_{50}^T$ for the matrix 
$T=(100;\delta,\sigma)$, where  $\sigma$ and $\delta$ are arbitrarily chosen real parameters.  The images are obtained by using the MATLAB {\rm imagesc} function, with {\rm colorbar}. }\label{E}
\end{figure}

\begin{figure}[tbp]
\centering
\includegraphics[scale=0.45]{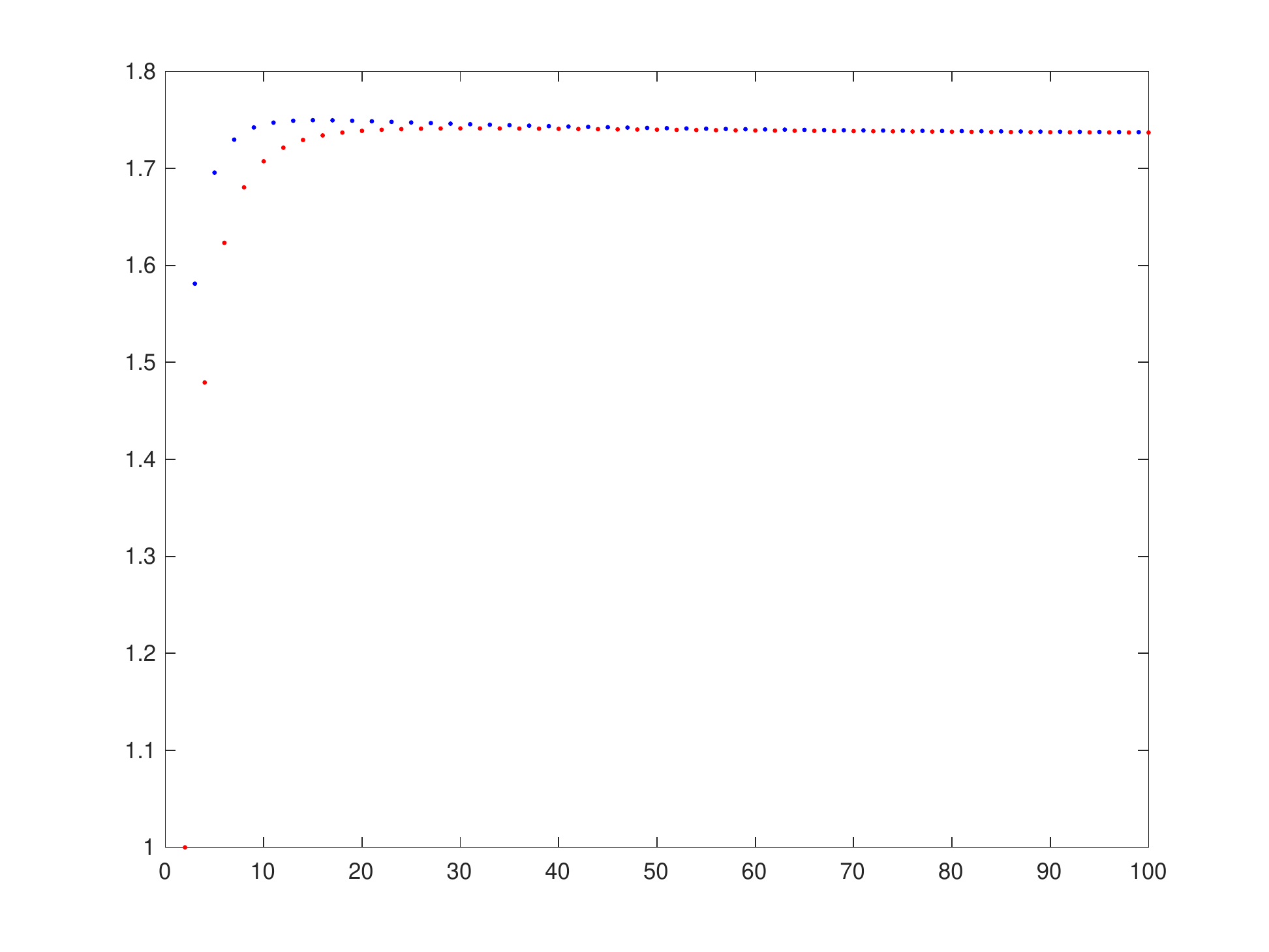}
\caption{Ratios between the largest and the smallest structured condition numbers for the matrix $T=(n;\delta,\sigma)$, where  $\sigma$ and $\delta$ are arbitrarily chosen real parameters and $n=2,\dots,100$. The horizontal axis shows the dimension $n$ of $T$ and the vertical axis the relevant  ratios $\kappa^{\cal T}(\lambda_{n})/\kappa^{\cal T}(\lambda_{n/2})$ in red, when $n$ is even, and $\kappa^{\cal T}(\lambda_{n})/\kappa^{\cal T}(\lambda_{(n+1)/2})$ in blue, when $n$ is odd.}\label{ratios}
\end{figure}

\section{An upper bound for the structured distance to singularity}\label{sec22}
The structured conditioning analysis in Section \ref{sec2}  can be exploited to obtain an upper bound for the structured distance to singularity of $T\in{\cal T}$ in the Frobenius norm.
\begin{theorem}\label{t_ansatz}
Consider  $T=(n;\delta,\sigma)$. One has that 
\begin{equation}\label{ansatz}
T-\lambda_h\frac{(\x_h\x_h^T) |{\cal T}}{\|(\x_h\x_h^T) |{\cal T}\|_{F}^2},
\end{equation}
where $\x_h$ is  a unit-norm eigenvector corresponding to $\lambda_h$, is a symmetric tridiagonal Toeplitz matrix with null $h$th eigenvalue. Its distance to $T$ in the Frobenius norm is
given by the ratio $| \lambda_h|/\kappa^{\cal T}(\lambda_h)$, where $\kappa^{\cal T}(\lambda_h)$ is the structured condition number of $\lambda_h$.
\end{theorem}
\begin{proof}
By \eqref{nrmsq}, one has
\[
\x_h^T\,\left(T-\lambda_h\frac{(\x_h\x_h^T) |{\cal T}}{\|(\x_h\x_h^T) |{\cal T}\|_{F}^2}\right)\,\x_h=
\x_h^T\,T\,\x_h-\frac{\lambda_h}{\|(\x_h\x_h^T) |{\cal T}\|_{F}^2}\x_h^T\,(\x_h\x_h^T) |{\cal T}\,\x_h
\]
\[
=\lambda_h-\frac{\lambda_h}{\|(\x_h\x_h^T) |{\cal T}\|_{F}^2}\|(\x_h\x_h^T) |{\cal T}\|_{F}^2=0.
\]
Thus, counting eigenvalues as in \eqref{lamh}, the $h$th eigenvalue of the symmetric tridiagonal Toeplitz matrix in \eqref{ansatz} is zero.  By \eqref{klam0}, one has
\[
\left\|\lambda_h\frac{(\x_h\x_h^T) |{\cal T}}{\|(\x_h\x_h^T) |{\cal T}\|_{F}^2}\right\|_F=\frac{| \lambda_h|}{\|(\x_h\x_h^T) |{\cal T}\|_{F}}=
\frac{| \lambda_h|}{\kappa^{\cal T}(\lambda_h)},
\]
which concludes the proof.
\end{proof}

We observe that Theorem \ref{ansatz} implies an upper bound for the structured distance to singularity of $T$ in the Frobenius norm, that is to say,
\[
\min_{h=1,\dots,n}|\lambda _{h}|= d_{F}(T)\leq d_{F}^{\cal T}(T)\leq 
\min_{h=1,\dots, n}\frac{| \lambda_h|}{\kappa^{\cal T}(\lambda_h)}.
\]

\section{Structured distance to singularity} \label{sec3}
According to Theorem \ref{t_ansatz}, the matrix \eqref{ansatz} belongs to ${\mathcal S}\cap{\mathcal T}$.
Is \eqref{ansatz} the matrix in ${\mathcal T}$ with null $h$th eigenvalue closest to $T$ in the Frobenius norm? The following results address such issue.

\begin{theorem}\label{t_sdstar}
Consider  $T=(n;\delta,\sigma)$. The unique symmetric  tridiagonal Toeplitz matrix with null $h$th eigenvalue  closest to $T$ in the Frobenius norm is $S_*^{\mathcal T}=(n;\delta_*,\sigma_*)$, where
\begin{equation}\label{sdstar}
\delta_*= \frac{2(nc^2\delta -(n-1)c\sigma)}{n-1+2nc^2}, \;\;  \;\; \sigma_*= \frac{(n-1)\sigma -nc\delta}{n-1+2nc^2}, \;\;\;\;{\mbox with}\;\;c=\cos \frac{h\pi}{n+1}.
\end{equation}
One has  
\begin{equation}\label{dFT2}
\|T-S_*^{\mathcal T}\|_F= \frac{| \delta + 2 c\sigma|}{\sqrt{\frac{1}{n} +\frac{2c^2}{n-1}}}.
\end{equation}

\end{theorem}

\begin{proof}
We seek to determine  a matrix $S=(n;d,s) \in {\mathcal T}$ such that its $h$th eigenvalue is $0$. By \eqref{lamh}, one has 
$d= -2s\cos \frac{h\pi}{n+1}$. Thus, $S=(n;-2cs,s)$, with $c=\cos \frac{h\pi}{n+1}$, and the squared distance to $T$ can
be expressed as
\begin{equation}\label{SmT}
\|T-S\|_F^2=2(2nc^2+n-1)s^2+4(nc\delta+( 1-n)\sigma)s+n\delta^2+2(n-1)\sigma^2.
\end{equation}
Differentiating \eqref{SmT} with respect to $s$ and equating to zero  straightforwardly yields the quantity  
$ \sigma_*$ in \eqref{sdstar}. Computing $-2c\sigma_*$ gives $ \delta_*$ in \eqref{sdstar}.  Finally,  replacing $s$  in \eqref{SmT} by $ \sigma_*$ in \eqref{sdstar}, dividing both numerator and denominator by $n(n-1)$, and taking the square root, gives \eqref{dFT2} and concludes the proof.
\end{proof}

\begin{corollary}\label{c_dFT3}
Consider $T=(n;\delta,\sigma)$.  The symmetric tridiagonal Toeplitz matrix \eqref{ansatz} is the unique matrix in ${\mathcal T}$ with null $h$th eigenvalue  closest to $T$ in the Frobenius norm.
\end{corollary}
\begin{proof}
One observes that the distance in \eqref{dFT2} coincides with the ratio between the absolute value of the eigenvalue $\lambda _{h}$ and its structured condition number $\kappa^{\cal T}(\lambda_h)$. Since Theorem \ref{t_ansatz} claims that the matrix in \eqref{ansatz} is a symmetric  tridiagonal Toeplitz matrix  with null $h$th eigenvalue at such distance to $T$ whereas   Theorem \ref{t_sdstar}  claims that $S_*^{\mathcal T}=(n;\delta_*,\sigma_*)$ is the unique matrix in ${\mathcal T}$  with null $h$th eigenvalue  closest to $T$ in the Frobenius norm, the proof is concluded.
\end{proof}

We are now in a position to determine $S_*^{\mathcal T}=\arg\min_{S \in {\mathcal S}\cap{\mathcal T} }\| T -  S\|_F $
and its distance to $T$, i.e., $d_F^{\cal T}(T)= \| T- S_*^{\mathcal T} \|_F$.
\begin{theorem}\label{fins}
Consider $T=(n;\delta,\sigma)$ and let $h\in\{1,\dots,n\}$ be such that 
\begin{equation*}
	\frac{|\lambda_h|}{\kappa^{\cal T}(\lambda_h)}\leq\frac{|\lambda_k|}{\kappa^{\cal T}(\lambda_k)}, \quad \;\;  k=1,\dots,n.
\end{equation*}
Then the symmetric tridiagonal matrix $S_*^{\mathcal T}=(n;\delta_*,\sigma_*)$ built as in Theorem \ref{t_sdstar} attains the minimum $\min_{S\in{\mathcal S}\cap{\mathcal T}} \|T-S\|_F$. Hence, the structured distance to singularity of $T$ is
\begin{equation}\label{strdist}
	d_F^{\cal T}(T)=\|T-S_*^{\mathcal T}\|_F=\frac{|\lambda_h|}{\kappa^{\cal T}(\lambda_h)}.
\end{equation}
Moreover, if 
\begin{equation}\label{strcond}
	\frac{|\lambda_h|}{\kappa^{\cal T}(\lambda_h)}<\frac{|\lambda_k|}{\kappa^{\cal T}(\lambda_k)}, \quad \;\;  k=1,\dots,n, \;\;k\ne h,
\end{equation}
then  $S_*^{\mathcal T}=(n;\delta_*,\sigma_*)$ is the unique closest matrix in ${\mathcal S}\cap{\mathcal T}$ to $T$.
\end{theorem}
\begin{remark}\label{rmknrm1}
Let  $d_2(T)$ denote the distance to singularity of $T=(n;\delta,\sigma)$ in the spectral norm. If $\lambda _{k}$ is such that $|\lambda_k|
<|\lambda_j|$, with $j=1,\dots, n, j\ne k$, i.e., if $|\lambda _{k}|$ is the unique smallest singular value of $T$, then
the matrix $S_*$  defined in \eqref{Sstar} is the unique matrix in ${\mathcal S}$ such that
\[
d_{2}(T)=\min_{S\in 
{\mathcal S}} \left\| T-S\right\|_{2}= \left\| T- S_* \right\|_{2}
=\left\| \lambda_k\x_k\x_k^T\right\|_{2}=|\lambda _{k}|,
\]
where $\|\cdot\|_2$ stands for the spectral norm.  Hence  $d_{2}(T)=d_{F}(T)$.
It is now interesting to observe that Theorem \ref{fins} gives an upper bound for the structured distance to singularity of $T$ in the spectral norm $d_{2}^{\cal T}(T)=\min_{S\in {\mathcal S}\cap{\mathcal T}} \left\| T-S\right\|_{2}\,$. 
 In fact, if 
$$
	\frac{|\lambda_h|}{\kappa^{\cal T}(\lambda_h)}\leq\frac{|\lambda_j|}{\kappa^{\cal T}(\lambda_j)}, \quad \;\;  j=1,\dots,n, 
$$
then $S_*^{\mathcal T}=(n;\delta_*,\sigma_*) \in {\mathcal S}\cap{\mathcal T}$, with $\delta_*$, $\sigma_*$, and $c$ in \eqref{sdstar}, is such that
$$
d_2^{\cal T}(T)\leq \|T-S_*^{\mathcal T}\|_2= \|X(\Lambda-\Lambda_*)X^T\|_2= \|\Lambda-\Lambda_*\|_2,
$$
where  we denote by $\Lambda_*$  the diagonal matrix having as $j$th diagonal entry the $j$th eigenvalue of $S_*^{\mathcal T}$, for $j=1,\dots,n$. Then, since
$$
 \|\Lambda-\Lambda_*\|_2=
 \frac{|\lambda_h|}{n-1+2nc^2}(n-1+2n|c|\cos\frac{\pi}{n+1}),
 $$
one has the following lower and upper bounds for $d_2^{\cal T}(T)$:
\begin{equation}\label{norm2}
|\lambda _{k}|=d_{2}(T)\leq d_2^{\cal T}(T)\leq  \frac{|\lambda_h|}{n-1+2n\cos^2 \frac{h\pi}{n+1}}(n-1+2n\left|\cos \frac{h\pi}{n+1}\right|\cos\frac{\pi}{n+1}).
\end{equation}
\end{remark}

\section{The definite case} \label{sec4}

If $T=(n;\delta,\sigma)$ is positive or negative definite, then the smallest eigenvalue in magnitude is either $\lambda_n$ or $\lambda_1$, according to the sign of $\sigma$. 
Following our analysis on the eigenvalue structured sensitivity, cf.  Section \ref{sec2},   the two extremal eigenvalues in the spectrum of $T$ have the largest structured condition numbers. Therefore, since there exists $k\in \{1,n\}$ such that $|\lambda_k|<|\lambda_h|$, for $h=1,\dots,n$, $h\ne k$, then the same eigenvalue $\lambda_k$ satisfies the inequalities in \eqref{strcond}. Hence, by Corollary \ref{c_dFT3}, the structured distance to  singularity  is given by 
the ratio between the absolute value of the smallest eigenvalue in magnitude and its structured condition number. 
\begin{remark}
Note that the approach here is the opposite of the one in \cite[Example 4.6]{NR21}, where the authors
were concerned with the {\it structured distance to symmetric positive 
semidefiniteness}  $\Delta_F^{\cal{T}}(T)$ of  (indefinite) matrices $T\in{\cal T}$. In more detail, by \cite[Theorem 4.5]{NR21}, one has
\[
\Delta_F^{\mathcal{T}}(T)\leq\min\sqrt{2(n-1)\sigma^2, n\max\left\{0,2|\sigma|\cos\frac{\pi}{n+1}- 
\delta\right\}^2}, \;\;\;\;\;\;\;\; \rm{if} \;\; \delta>0,
\]
\[
\Delta_F^{\mathcal{T}}(T)\leq\min\sqrt{2(n-1)\sigma^2+n\delta^2,n\left(2|\sigma|
\cos\frac{\pi}{n+1}-\delta\right)^2}, \;\;\;\;\;\;\;\; \rm{if} \;\; \delta\leq0 .
\]
\end{remark}

\subsection{About the the structured distance to singularity in the spectral norm} 
Remark \ref{rmknrm1} shows lower and upper bounds for the structured distance to singularity in the spectral norm, $d_2^{\cal T}(T)$,  of any symmetric tridiagonal Toeplitz matrix $T=(n;\delta,\sigma)$.
In the definite case,  such structured distance equals the (unstructured) distance to singularity $d_2(T)$. Indeed, since the $k$th eigenvalue of $T$,  with either $k=1$ or $k=n$, is such that 
\begin{equation*}
|\lambda_k|<|\lambda_h| \quad \mbox{and}\;\;	\frac{|\lambda_k|}{\kappa^{\cal T}(\lambda_k)}<\frac{|\lambda_h|}{\kappa^{\cal T}(\lambda_h)}, \quad \mbox{with}\;\;  h=1,\dots,n,\;\; h\ne k,
\end{equation*}
then by  \eqref{norm2} one has 
$$
|\lambda_k|=d_2(T)\leq d_2^{\cal T}(T)\leq  \frac{|\lambda_k|}{n-1+2n\cos^2\frac{\pi}{n+1}}(n-1+2n\cos^2\frac{\pi}{n+1})=|\lambda_k|,
$$
so that
$$
d_2(T)= d_2^{\cal T}(T)=|\lambda_k|=d_F(T)<d_F^{\cal T}(T)=\frac{|\lambda_k|}{\kappa^{\cal T}(\lambda_k)}. 
$$
We emphasize that we have shown that the equality 
$$d_2^{\cal T}(T)=\frac{1}{\|T^{-1}\|_2}$$
holds true if and only if $T=(n;\delta,\sigma)$ is positive or negative definite. Thus,  following the analysis in the work of Rump \cite{R12003}, in such a case one has that the reciprocal of the spectral condition number of $T$  is equal to its relative structured distance to the closest singular matrix with respect to the the spectral norm:
\[
\frac{1}{\kappa_2(T)}=\min_{\Delta T \in {\cal T}} \left\{\frac{\|\Delta T\|_2}{\|T\|_2} : T+\Delta T \in {\cal S} \right\}.
\]
\subsection{The discrete Laplacian} 
Eq. \eqref{n3} gives the following estimate for  the ratio between the unstructured  
and the structured distances to singularity of a large definite matrix $T\in {\cal T}$,
\[
\frac{d_{F}(T)}{d_{F}^{\cal T}(T)}\approx\sqrt{\frac{3}{n}}.
\] 

Consider as an example the discrete Laplacian $T(n)=(n; 2,-1)$. One observes that
$\lambda_1(n)$ is the smallest eigenvalue of $T(n)$ and, in a asymptotical perspective, one has
 $$
 \lambda_1(n)=2-2\cos \frac{\pi}{n+1}=\pi^2\cdot \frac{1} {n^2}\left(1+{\cal O}\left(\frac{1}{n}\right)\right)\,,
$$
whereas
$$
\frac{\lambda_1(n)}{\kappa^{\cal T}(\lambda_1(n))}=\frac{\lambda_1(n)}{\sqrt{\frac{1}{n} +\frac{2}{n-1}\cos^2 \frac{\pi}{n+1}}}=\frac{\pi^2}{\sqrt{3}}\cdot \frac{1}{n^{\frac 3 2}}\left(1+{\cal O}\left(\frac{1}{n}\right)\right)\,.
$$
Thus, as it is clear by measuring its structured (or unstructured) distance to singularity either in the spectral norm or in the Frobenius norm, $T(n)$ is asymptotically singular. In more detail, since in Theorem \ref{fins} one has $h=1$,  the closest matrix in ${\mathcal S}\cap{\mathcal T}$ to $T(n)$ in the Frobenius norm is $S_*^{\mathcal T}(n)=(n;\delta_*(n),\sigma_*(n))$, with $\delta_*(n)$, $\sigma_*(n)$ as in \eqref{sdstar} and  $c(n)=\cos \frac{\pi}{n+1}$, so that
$$
\lim\limits_{n\to \infty}c(n)=\lim\limits_{n\to \infty}\cos \frac{\pi}{n+1}=1.
$$
Thus, the entries of $S_*^{\mathcal T}(n)$ satisfy
\begin{equation*}
\lim\limits_{n\to \infty}\delta_*(n)= \lim\limits_{n\to \infty}\frac{2(2nc(n)^2+(n-1)c(n))}{2nc(n)^2+n-1}=2,
\end{equation*}
\begin{equation*}
\lim\limits_{n\to \infty}\sigma_*(n)= -\lim\limits_{n\to \infty}\frac{(n-1)+2nc(n)}{n-1+2nc(n)^2}=-1.
\end{equation*}
For the sake of completeness, we remark that an explicit formula for the inverse of $T(n)$ is given in \cite[Theorem 2.8]{Me}:
\[
[T^{-1}(n)]_{i,j}=i\,\frac{n-j+1}{n+1}. 
\]

\subsection{About the pattern of the Cholesky factor}\label{cholf}
In \cite[Lemma 1]{LMV}, Laudadio et al. showed that the diagonal entries of the upper triangular factor $R$ of the Cholesky factorization  of a symmetric positive definite Toeplitz matrix $A=R^TR$  decrease monotonically with increasing row number. We are interested in further  monotonicity properties of the entries of the Cholesky factor $R$  of a symmetric positive definite tridiagonal Toeplitz matrix, which can be easily obtained and will be useful to illustrate the numerical tests regarding positive definite matrices in ${\cal T}$, cf. Example 1 in Section \ref{sec5}. We collect  such results in the following proposition.

\begin{proposition}\label{t_mon1}
Let $T=(n;\delta,\sigma)$, $\sigma\ne0$,  be positive definite and let $R=(r_{i,j})$, $i,j=1,\dots,n$ be its Cholesky factor so that $T=R^TR$. Then, $R$ is upper bidiagonal with the entries of the diagonal  satisfying
\begin{equation}\label{ri}
r_{i-1,i-1}\geq r_{i,i}>0, \;\; i=2,\dots,n\,,
\end{equation}
and the entries of the upper diagonal having the same sign of $\sigma$ and satisfying
\begin{equation}\label{absineq}
|r_{i-1,i}|\leq |r_{i,i+1}|, \;\; i=2,\dots,n-1\,.
\end{equation}
Moreover, if $\delta\geq 2|\sigma|$, one has
\begin{equation}\label{last}
r_{i-1,i-1}>|r_{i-1,i}| \;\;\;\;\;\; \mbox{and} \;\;\;\;\;\; r_{i,i}> |r_{i-1,i}|,  \;\;\;\;\;\; i=2,\dots,n\,.
\end{equation}
\end{proposition}

\begin{proof}
It is well known that the Cholesky factor of a symmetric positive definite tridiagonal matrix is upper bidiagonal. The inequalities in \eqref{ri} are proved  in \cite[Lemma 1]{LMV}.  The entries of $T$  satisfy 
\begin{equation}\label{sigma}
\delta=r_{i,i}^2+r_{i-1,i}^2\;\;\;\;\;\; \mbox{and} \;\;\;\;\;\;\sigma=r_{i-1,i}r_{i-1,i-1},  \;\;\;\;\;\;\;\;\; i=2,\dots,n\,.
\end{equation}
The equalities satisfied by $\sigma$ in \eqref{sigma} imply that all the entries of the upper diagonal of $R$ have the same sign as $\sigma$  and that
$$
|r_{i-1,i}|=|r_{i,i+1}|\frac{r_{i,i}}{r_{i-1,i-1}}\leq |r_{i,i+1}|, \;\; i=2,\dots,n-1\,
$$
where the inequality is due to \eqref{ri}. This concludes the first part of the proof.

Let $\delta$ and $\sigma$ satisfy $\delta\geq 2|\sigma|>0$. Firstly, we prove that
\begin{equation}\label{ru}
r_{i-1,i-1}> |r_{i-1,i}| \;\; \Longleftrightarrow \;\; r_{i,i}> |r_{i-1,i}|, \;\;\;\;\;\;\;\; i=2,\dots,n\,.
\end{equation}
Assume that $r_{i-1,i-1}> |r_{i-1,i}| $. Thanks to \eqref{sigma}, one has
\[
r_{i,i}^2=\delta-r_{i-1,i}^2\geq 2|\sigma|-r_{i-1,i}^2=2|r_{i-1,i}| r_{i-1,i-1}-r_{i-1,i}^2> 2r_{i-1,i}^2-r_{i-1,i}^2=r_{i-1,i}^2,
\]
where in the last inequality we have used that $r_{i-1,i-1}> |r_{i-1,i}|$. This gives  $r_{i,i}> |r_{i-1,i}|$. The viceversa is 
assured  by \eqref{ri}. Then, we prove
\begin{equation}\label{ur}
r_{i,i}>|r_{i-1,i}| \;\;  \Longleftrightarrow\;\; r_{i,i}>|r_{i,i+1}|, \;\;\;\;\;\;\;\; i=2,\dots,n-1\,.
\end{equation}
Assume that $r_{i,i}> |r_{i-1,i}|$. By  \eqref{sigma}, one has
\[
r_{i,i}^2=\frac{r_{i,i}^2+r_{i,i}^2}{2}>\frac{r_{i,i}^2+r_{i-1,i}^2}{2}=\frac{\delta}{2}\geq{|\sigma|}=|r_{i,i+1}|r_{i,i},
\]
This gives  $r_{i,i}> |r_{i,i+1}|$. The viceversa is assured  by \eqref{absineq}. In order to prove \eqref{last}, we observe  that $r_{1,1}\geq 2\,|r_{1,2}|$. Indeed, $r_{1,1}^2=\delta$, by construction, so that, by \eqref{sigma},  one has $r_{1,1}^2\geq 2|\sigma|=2\, |r_{1,2}|r_{1,1}$, which implies that $r_{1,1}\geq 2|r_{1,2}|$. Now, the proof follows by recursively applying the  $\Longrightarrow$ implications  in Eqs. \eqref{ru} and  \eqref{ur},  from $r_{1,1}> |r_{1,2}| \;\;  \Longrightarrow\;\; r_{2,2}> |r_{1,2}|$ till $r_{n-1,n-1}> |r_{n-1,n}| \;\;  \Longrightarrow\;\; r_{n,n}> |r_{n-1,n}|$. This concludes the proof.
\end{proof}

\section{The indefinite case} \label{secx}
We are interested in analyzing  the cases when the closest matrix in ${\mathcal S}\cap{\mathcal T}$ to T in the Frobenius norm is not unique. As we commented above, this can happen only in the indefinite case.
 
 \subsection{The case $\boldsymbol{\delta}\, \mathbf{=0}$} \label{subsec22} 
 
We have to distinguish two cases.
If $\delta=0$ and $n$ is odd, then $T$ is singular ($d_{F}(T)=d_{F}^{\cal T}(T)=0$), whereas
if $\delta=0$ and $n$ is even, then the eigenvalues $\lambda_{n/2}$ and $\lambda_{(n+2)/2}$ have opposite signs and the same smallest absolute value, because 
$\cos \frac {n \pi}{2(n+1)}=-\cos \frac {(n+2) \pi}{2(n+1)}$.
In the latter case, 
\[
d_{F}(T)=| \lambda_{n/2}|=| \lambda_{(n+2)/2}|
\]
 and, by the analysis in Section \ref{sec2}, one proves that
\[
d_{F}^{\cal T}(T)=\frac{| \lambda_{n/2}|}{\kappa^{\cal T}(\lambda_{n/2})}=\frac{| \lambda_{(n+2)/2}|}{\kappa^{\cal T}(\lambda_{(n+2)/2})}.
\]
Indeed, if $\delta=0$,  then $ |\lambda_{k}|\leq|\lambda_{h}|$
implies 
$ |\lambda_{k}|/\kappa^{\cal T}(\lambda_{k})\leq |\lambda_{h}|/\kappa^{\cal T}(\lambda_{h}),$
and the equalities hold only for $\lambda_{n/2}$ and $\lambda_{(n+2)/2}=-\lambda_{n/2}$.
Thus, there are two closest matrices in ${\mathcal S}$ to $T$ in the Frobenius norm, i.e.,
\[
T\pm\lambda_{n/2}(\x_{n/2}\x_{n/2}^T),
\]
and  two closest matrices in ${\mathcal S}\cap{\mathcal T}$ to $T$ in the Frobenius norm, i.e.,
\[
T\pm\lambda_{n/2}\frac{(\x_{n/2}\x_{n/2}^T) |{\cal T}}{\|(\x_{n/2}\x_{n/2}^T) |{\cal T}\|_{F}^2}.
\]
Eq. \eqref{n2} gives, when $n$ is large, the following estimate for  the ratio between the unstructured  
and the structured distances to singularity of $T$ 
\[
\frac{d_{F}(T)}{d_{F}^{\cal T}(T)}\approx\sqrt{\frac{1}{n}}.
\] 

Consider as an example $T(n)=(n;0,\sigma)\in{\cal T}$, with $\sigma>0$. One observes that, setting $c(n)=\cos \frac{n\pi}{2(n+1)}$
for the smallest eigenvalue of $T(n)$, in a asymptotical perspective, one has
 $$
\lambda_{\frac n 2}=-\lambda_{\frac {n+2} {2}}=2\sigma c(n)= {\cal O}\left(\frac{1} {n}\right),
$$
whereas
$$
\frac{ |\lambda_{\frac n 2}|}{\kappa^{\cal T}(\lambda_{\frac n 2})}=\frac{\lambda_{\frac n 2}}{\kappa^{\cal T}(\lambda_{\frac n 2})}=-\frac{\lambda_{\frac {n+2} {2}}}{\kappa^{\cal T}(\lambda_{\frac {n+2} {2}})}=   \frac{2\sigma c(n)}{\sqrt{\frac{1}{n} +\frac{2}{n-1}c(n)^2}} = {\cal O}\left(\frac{1} {\sqrt{n}}\right)\,.
$$
On the contrary, if one has $\sigma<0$, then the signs of $\lambda_{\frac n 2}$ and $\lambda_{\frac {n+2} {2}}$ are the opposite, however the conclusion is the same:
$T(n)$ is asymptotically singular. Its (unstructured) distance to singularity is infinitesimal of order $\frac{1} {n}$ whereas its structured distance to singularity is an infinitesimal of order $\frac{1} {\sqrt{n}}$. Moreover, since
$
\lim\limits_{n\to \infty} \pm c(n)=0,
$
the entries of both the closest matrices in ${\mathcal S}\cap{\mathcal T}$ to $T$ in the Frobenius norm, satisfy
\begin{equation*}
\lim\limits_{n\to \infty}\delta_*(n)^{\pm}= \mp\lim\limits_{n\to \infty}\frac{2(n-1)c(n)\sigma }{n-1+2nc(n)^2}=0,
\end{equation*}
\begin{equation*}
\lim\limits_{n\to \infty}\sigma_*(n)^{\pm}= \lim\limits_{n\to \infty}\frac{(n-1)\sigma}{n-1+2nc(n)^2}=\sigma.
\end{equation*}

\subsection{The case $\boldsymbol{\delta}\, \mathbf{\ne 0}$}  \label{subsec23}

If the diagonal entry of $T$ in not trivial  and there is not a unique eigenvalue with the smallest magnitude (i.e., if the property required by Eckart-Young 
 theorem is not satisfied), then on the contrary  there exists a unique closest matrix in ${\mathcal S}\cap{\mathcal T}$ to $T$ in the Frobenius norm. This is good news. However, we are interested in analyzing the opposite case where there exists a unique closest matrix in $\mathcal S$ to $T$ in the Frobenius norm but \eqref{strcond} is not satisfied.  
 
 It is straightforward that if  $\lambda_{k_1}$ and $\lambda_{k_2}$ are two eigenvalues with the smallest magnitude, then they are of opposite sign and $k_1=k_2\pm1$. Moreover, one can observe that the same happens also for the  eigenvalues $\lambda_{k_1}$ and $\lambda_{k_2}$, with $|\lambda_{k_1}|\ne|\lambda_{k_2}|$,   satisfying 
\begin{equation}\label{eqq2}
\frac{|\lambda_{k_1}|}{\kappa^{\cal T}(\lambda_{k_1})}= \frac{|\lambda_{k_2}|}{\kappa^{\cal T}(\lambda_{k_2})}<\frac{|\lambda_h|}{\kappa^{\cal T}(\lambda_h)}, \quad \mbox{with}\;\;  h=1,\dots,n,\; \mbox{and}\;\; h\notin \{ k_1,k_2 \}.
\end{equation}
In more detail,  since the sequence $\{|\lambda_h|\}_{h=1,\dots n}$ first decreases and then increases, then:

\noindent -- if $n$ is odd,  there can exist only two eigenvalues $\lambda_{k_1}$ and $\lambda_{k_2}$ of opposite sign, with consecutive indices  both belonging either to $\{1,\dots n_2\}$ or  to $\{n_2,\dots n\}$, where $n_2$  denotes the smallest integer greater than $n/2$, that satisfy \eqref{eqq2}, because the sequence $\{1/\kappa_F^{\cal T}(\lambda_h)\}$ increases for $h=1,\dots n_2$  and decreases for $h=n_2,\dots n$;

\noindent -- if $n$ is even,  there can exist only two eigenvalues $\lambda_{k_1}$ and $\lambda_{k_2}$ of opposite sign, with consecutive indices  both belonging either to $\{1,\dots n/2\}$ or  to $\{n/2+1,\dots n\}$, that satisfy \eqref{eqq2}, because the sequence  $\{1/\kappa_F^{\cal T}(\lambda_h)\}$ increases for $h=1,\dots n/2$ and decreases for $h=n/2+1,\dots n$.

Finally, it can happen that both the closest matrices in ${\mathcal S}$ and in ${\mathcal S}\cap{\mathcal T}$ are unique whereas the eigenvalue of smallest magnitude is not the same as the eigenvalue satisfying \eqref{strcond}.
\section{Numerical examples} \label{sec5}
This section provides four representative examples for the cases that have been analyzed throughout the paper and also reports remarks about the monotonicity properties of the entries of the Cholesky factor of positive definite matrices in $\cal T$ that have been investigated in Subsection \ref{cholf} (cf. Remark \ref{pat}) as well as  a quantitative test designed to illustrate the cases relevant to  indefinite matrices in $\cal T$ that have been investigated in Subsection \ref{subsec23} (cf. Remark \ref{last}). The four examples mentioned above address  the case of  positive definite test matrices (Example 1), the case where there is neither a unique closest matrix in ${\mathcal S}$ nor in ${\mathcal S}\cap{\mathcal T}$ to the test matrix (Example 2),  the case where the smallest eigenvalue in magnitude is not unique but there is an eigenvalue that satisfies \eqref{strcond} (Example 3), and the case where the unique smallest eigenvalue in magnitude is not the same as the unique eigenvalue that satisfies \eqref{strcond} (Example 4).  
Since the eigenvalues and eigenvectors of symmetric tridiagonal Toeplitz matrices are known in closed form, and all ingredients of our analysis are easily computable, assessing the theoretical results has been straightforward. All computations were carried out in MATLAB R2022b  with about $16$ significant decimal digits on an iMac with a 3,2 GHz Intel Core i7 6 core and equipped with 16 GB of RAM.

\medskip

\noindent {\bf Example 1.}
Consider  $T=(1000; 2,-1)$. This test matrix is, as it is well known, symmetric positive definite.
The smallest eigenvalue is $\lambda_{1}$,  hence $d_{F}(T)=\lambda_{1}=9.8499 \cdot 10^{-6}$. 
Since $\kappa^{\cal T}(\lambda_{1})$ results to be equal to $5.4790 \cdot 10^{-2}$ (close to $\sqrt{3/1000}=5.4772 \cdot 10^{-2}$), one has
$d_{F}^{\cal T}(T)=\lambda_{1}/\kappa^{\cal T}(\lambda_{1})=1.7977 \cdot 10^{-4}$.  As for the structured distance of $T$ to singularity in the spectral norm, one has  $d_{2}(T)=d_{2}^{\cal T}(T)=\lambda_{1}=9.8499 \cdot 10^{-6}$.
The left plot of Figure \ref{f_SPD1} shows the inverse of $R$, where $R$ is the Cholesky factor of $T$, whereas the right plot shows the image of the (centro-symmetric) inverse of $T$, i.e., $T^{-1}= R^{-1}R^{-T}$. The observed pattern of $R^{-1}$ may be considered an inheritance of the monotonicity properties satisfied by  the entries of $R$, as illustrated in Proposition \ref{t_mon1}, which in turn determine monotonicity properties for the entries of $R^{-1}$. Indeed, as shown in the left plot of Figure \ref{f_SPD1}, in the upper triangular matrix $R^{-1}$, all the structure entries  are positive and, in each row, monotonically decreasing with increasing column number. Moreover the entries of each structure diagonal are seen to be monotonically increasing with increasing row number. 
\medskip

\begin{remark}\label{pat}
We report that in all tests with symmetric positive definite tridiagonal Toeplitz matrices with $\sigma<0$, the entries of the extra-diagonals of the inverse of the Cholesky factor are  positive and with the same monotonicity properties as in Example 1, whereas, in case of $\sigma>0$, such monotonicity properties apply to the absolute values of the entries.  Finally, in line with the results about the well known issue of the exponential off-diagonal decay behavior occurring for the entries of the inverse of a symmetric positive definite tridiagonal Toeplitz matrix $T$ (see, e.g., \cite{DMS,EP,Me}),  we report that in all tests one has that the larger the (un)structured distance to singularity of $T$ is, the more the large entries of the inverse matrix $T^{-1}$ - or,  in case of  $\sigma>0$, the large entries in magnitude of the latter -  tend to squeeze towards the diagonal.
\end{remark}

\begin{figure}[tbp]\
\centering
\includegraphics[scale=0.34]{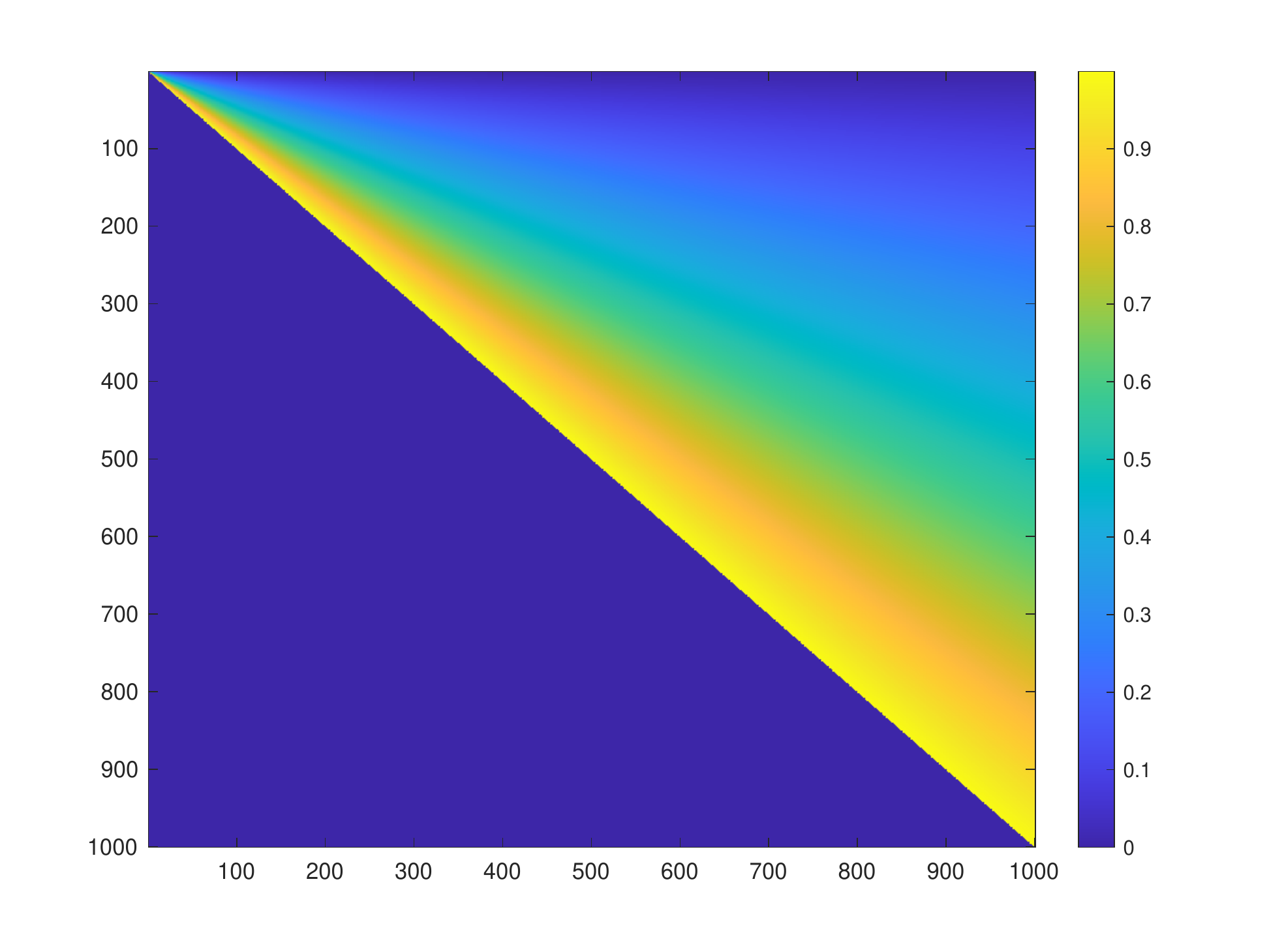}\includegraphics[scale=0.34]{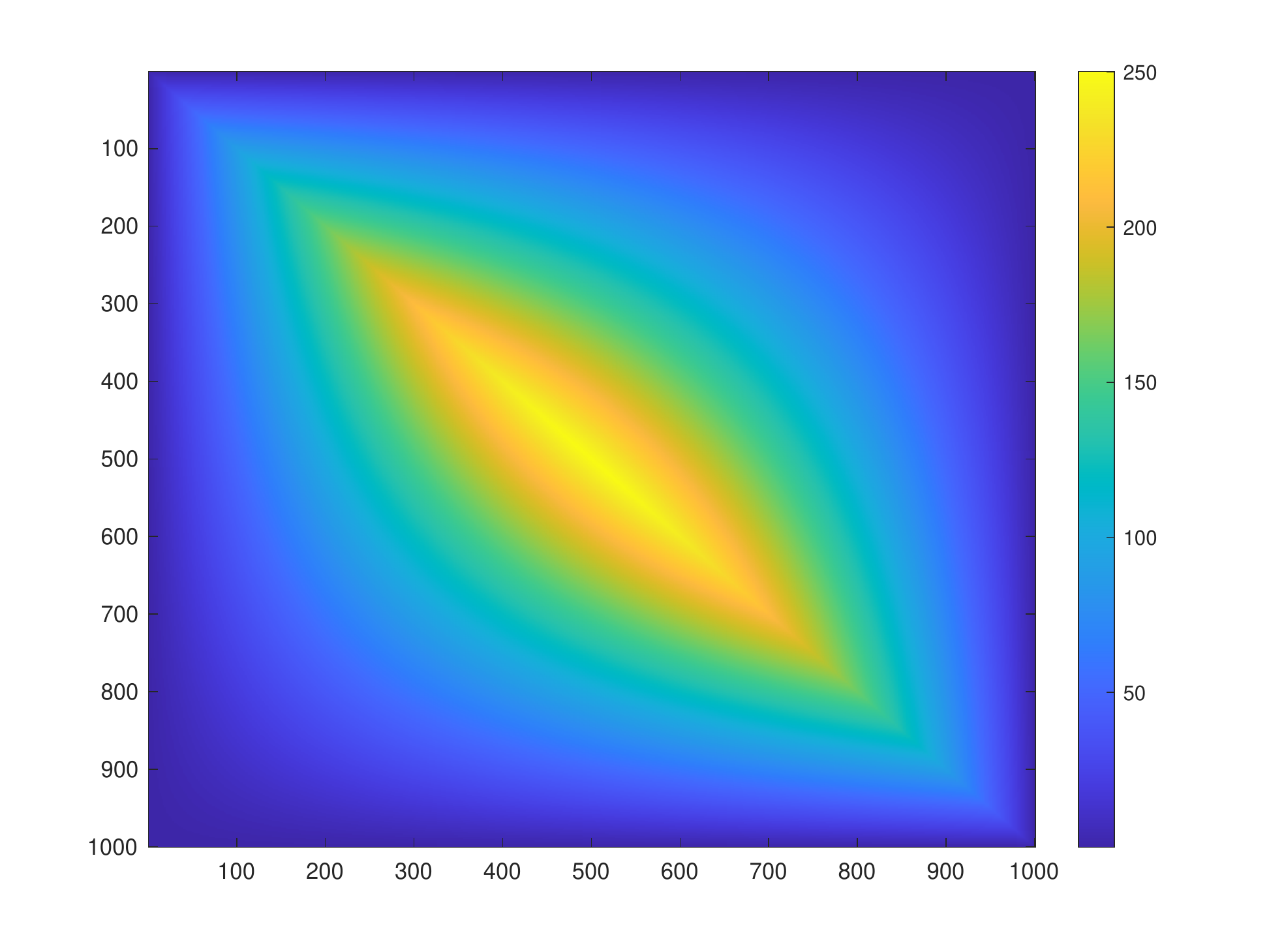}
\caption{Example 1. Images of  $R^{-1}$ and $T^{-1}$, obtained by using the MATLAB {\rm imagesc} function, with {\rm colorbar}.}\label{f_SPD1}
\end{figure}

\medskip
 
\noindent {\bf Example 2.}
Consider the symmetric tridiagonal Toeplitz matrix $T=(1000; 0,1)$.
The eigenvalues $\lambda_{500}$ and $\lambda_{501}$ are both the smallest in magnitude, and the distance to singularity is given by  $d_{F}(T)=\lambda_{500}=-\lambda_{501}=3.1385\cdot 10^{-3}$. Thus, there are two closest matrices in ${\mathcal S}$ to $T$.
As for the structured condition numbers, one has  $\kappa^{\cal T}(\lambda_{500})= \kappa^{\cal T}(\lambda_{501})=3.16229\cdot 10^{-2}$ (which is approximatively equal to $\sqrt{1/1000}=3.16228 \cdot 10^{-2}$), so that the structured distance to singularity is given by
$d_{F}^{\cal T}(T)=\lambda_{500}/\kappa^{\cal T}(\lambda_{500})=-\lambda_{501}/\kappa^{\cal T}(\lambda_{501})=9.9246\cdot 10^{-2}$.  Also, there are two closest singular symmetric tridiagonal Toeplitz matrices, as commented in Subsection \ref{subsec22}.

\medskip
 
\noindent {\bf Example 3.} 
Consider $T=(9; \cos \frac{\pi}{20},-\sqrt{2}/2)$. In this example, $n$, $\sigma$ and $\delta$ are chosen in a way that $T$ has two eigenvalues equally distant to the origin. 
In detail,  the eigenvalues $\lambda_2 =-1.5643\cdot 10^{-1}$ and $\lambda_3=1.5643\cdot 10^{-1}$   are the smallest in magnitude (and opposite in sign). Thus, $d_{F}(T)=1.5643\cdot 10^{-1}$. Table \ref{tab1} displays the eigenvalues and their structured      
condition numbers in the second and third columns. The smallest structured condition number is  relevant to the central eigenvalue  $\lambda_5=\cos \frac{\pi}{20}= 9.8769\cdot 10^{-1}$, with $\kappa^{\cal T}(\lambda_5)=\sqrt{1/9}=3.3333\cdot 10^{-1}$.  The farther from $\lambda_5$, the larger the structured condition numbers. 
Also, Table  \ref{tab1} displays the eigenvalues of both structured and unstructured closest singular matrices. In detail, the fourth column shows the eigenvalues of $S_*^{\cal T}$ and the fifth column displays the corresponding eigenvalues of $S_*=\sum_{h=1,\dots, n, h\ne 2}\lambda _{h}\x_h\x_h^T$ (we choose $k=2$ because $\lambda_2$ is the unique eigenvalue to satisfy also \eqref{strcond}). The latter eigenvalues are  seen to coincide with  the eigenvalues $\lambda_h$ in the second column, for any $h=1,\dots,9$, $h\ne 2$. In both the fourth and fifth columns the ``zeroed'' eigenvalue (i.e., the second one) is shown in bold face.
\begin{table}[bt]
\begin{center}
\begin{tabular}{ccccc}
$h$ & $\lambda_h$ & $\kappa^{\cal T}(\lambda_h)$ & $\lambda_h(S_*^{\cal T})$& $\lambda_h(S_*)$  \\  
\hline
$1$ & $-3.5731\cdot 10^{-1}$ & $5.8072\cdot 10^{-1}$& $-1.8452\cdot 10^{-1}$  & $-3.5731\cdot 10^{-1}$ \\  
$2$ & $-1.5643\cdot 10^{-1}$ & $5.2415\cdot 10^{-1}$& $\phantom{22}{\bf 2.2204\cdot 10^{-16}}$ & $\phantom{2}{\bf -5.4879\cdot 10^{-17}}$ \\  
$3$ & $\phantom{2}1.5643\cdot 10^{-1}$ & $4.4439\cdot 10^{-1}$& $\phantom{2}2.8739\cdot 10^{-1}$ & $\phantom{2}1.5643\cdot 10^{-1}$ \\  
$4$ & $\phantom{2}5.5067\cdot 10^{-1}$ & $3.6740\cdot 10^{-1}$ & $\phantom{2}6.4953\cdot 10^{-1}$ & $\phantom{2}5.5067\cdot 10^{-1}$ \\ 
$5$ & $\phantom{2}9.8769\cdot 10^{-1}$ & $3.3333\cdot 10^{-1}$ & $\phantom{2}1.0510\cdot 10^{-1}$  & $\phantom{2}9.8769\cdot 10^{-1}$ \\ 
$6$ & $1.4247\cdot 10^{0}$ & $3.6740\cdot 10^{-1}$  & $1.4524\cdot 10^{0}$   & $1.4247\cdot 10^{0}$ \\ 
$7$ & $1.8189\cdot 10^{0}$ & $4.4439\cdot 10^{-1}$ & $1.8145\cdot 10^{0}$  & $1.8189\cdot 10^{0}$ \\   
$8$ & $2.1318\cdot 10^{0}$ & $5.2415\cdot 10^{-1}$& $2.1019\cdot 10^{0}$ & $2.1318\cdot 10^{0}$ \\  
$9$ & $2.3327\cdot 10^{0}$ & $5.8072\cdot 10^{-1}$  & $2.2864\cdot 10^{0}$   & $2.3327\cdot 10^{0}$ \\  
\end{tabular}
\end{center}
\caption{Example 3. Eigenvalues $\lambda_h$ and their structured 
condition numbers $\kappa^{\cal T}(\lambda_h)$
for the matrix $T=(9; \delta, \sigma)$, where $\sigma = -\frac{\sqrt{2}}{2}$
and $\delta =\cos \frac{\pi}{20}$. Eigenvalues  $\lambda_h(S_*^{\cal T})$  and $\lambda_h(S_*)$, where $S_*^{\cal T}\in {\cal T}$ and $S_*$ are the projections of $T$ in ${\cal S}$ and ${\mathcal S}\cap{\mathcal T}$, respectively. The eigenvalues $\lambda_h(S_*)$ have been computed by using the MATLAB {\rm eig} function and numbered in increasing order to be compared both with $\lambda_h$ and $\lambda_h(S_*^{\cal T})$, $i=1,\dots,9$.}\label{tab1}
\end{table}

\medskip
 
\noindent {\bf Example 4.}
Consider  the symmetric tridiagonal Toeplitz matrix $T=(10; 1.8,-1)$.  The  unique eigenvalue smallest in magnitude is $\lambda_2=1.1749\cdot 10^{-1}$ while the  unique eigenvalue for which is smallest the ratio in \eqref{strcond} is $\lambda_1= -1.1899\cdot 10^{-1}$. Table  \ref{tabx} displays the eigenvalues $\lambda_h$ of $T$ and the relevant ratios $| \lambda_h|/\kappa^{\cal T}(\lambda_h)$. The projections of $T$ in ${\cal S}$ and ${\mathcal S}\cap{\mathcal T}$ are at distance $d_F(T)=1.1749\cdot 10^{-1}$ and $d_F^{\cal T}(T)=2.1560\cdot 10^{-1}$, respectively. 

\begin{table}[bt]
\begin{center}
\begin{tabular}{ccc}
$h$ & $\lambda_h$& $| \lambda_h|/\kappa^{\cal T}(\lambda_h)$  \\ 
\hline
$\phantom{1}1$ & $-1.1899\cdot 10^{-1}$  & $\phantom{1} 2.1560\cdot 10^{-1}$ \\ 
$\phantom{1}2$ & $\phantom{2}1.1749\cdot 10^{-1}$ & $\phantom{1} 2.3164\cdot 10^{-1}$ \\ 
$\phantom{1}3$ & $\phantom{2}4.9028\cdot 10^{-1}$ & $1.1094\cdot 10^{0}$ \\  
$\phantom{1}4$ & $\phantom{2}9.6917\cdot 10^{-1}$ & $2.6056\cdot 10^{0}$ \\ 
$\phantom{1}5$ & $1.5154\cdot 10^{0}$  & $4.6877\cdot 10^{0}$ \\  
$\phantom{1}6$ & $2.0846\cdot 10^{0}$   & $ 6.4487\cdot 10^{0}$ \\ 
$\phantom{1}7$ & $2.6308\cdot 10^{0}$  & $7.0730\cdot 10^{0}$ \\  
$\phantom{1}8$ & $ 3.1097\cdot 10^{0}$ & $7.0368\cdot 10^{0}$ \\ 
$\phantom{1}9$ & $3.4825\cdot 10^{0}$   & $6.8659\cdot 10^{0}$ \\  
$10$ & $3.7190\cdot 10^{0}$   & $6.7386\cdot 10^{0}$ \\  
\end{tabular}
\end{center}
\caption{Example 4. Eigenvalues  $\lambda_h$  and ratios $| \lambda_h|/\kappa^{\cal T}(\lambda_h)$.} \label{tabx}
\end{table}
 \begin{figure}[tbp]
\centering
\includegraphics[scale=0.35]{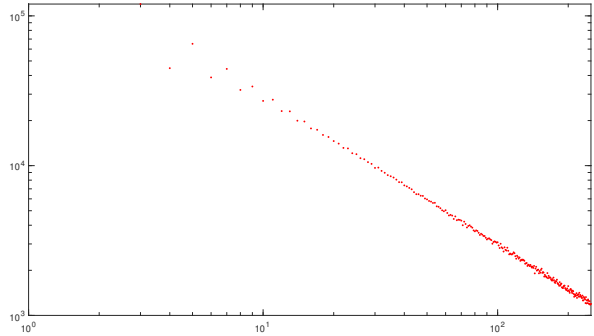}
\caption{Quantitative test (cf. Remark \ref{last}). In y-axis (in logarithmic-scale) the number of symmetric tridiagonal Toeplitz matrices of dimension $n$ where the unique eigenvalue of smallest magnitude is not the same as the  eigenvalue that satisfies \eqref{strcond}. In the x-axis  (in logarithmic-scale) the matrix dimension $n$.} \label{test}
\end{figure}

\medskip
\begin{remark}\label{last}
In order to have a quantitative idea of the incidence of cases where both the closest matrices in ${\mathcal S}$ and in ${\mathcal S}\cap{\mathcal T}$ are unique whereas the eigenvalue of smallest magnitude is not the same as the eigenvalue satisfying \eqref{strcond}, we tested $10^6$ indefinite matrices  $T=(n;\delta,\sigma)$, with $\delta\sigma\ne0$, for each dimension $n=2,\dots, 250$. Thus, the number of tested matrices  is almost $2.5\cdot 10^8$. 
In the investigation, $ \sigma$ and $\delta$  are random scalars drawn from the standard normal distribution at the condition that the cases where either $\delta\sigma=0$ or  $|\delta| > 2|\sigma|\cos(\pi/(n+1))$ are discarded. 
For any dimension $n$, not taking into account  the cases where there exist two eigenvalues having smallest magnitude or satisfying \eqref{eqq2}, we count the cases where the unique eigenvalue of smallest magnitude is not the same as the  eigenvalue that satisfies \eqref{strcond}.
The percentage of the such cases is  $0.4\%$ out of the total matrices tested. 
In Figure \ref{test}, we illustrate the trend of their incidence, with respect to the growing matrix dimension, by plotting  the relevant counter. 
\end{remark}

\section{Conclusions and future work} \label{sec6}
Given a real symmetric tridiagonal Toeplitz matrix $T$, this paper discusses the determination of its projection in the manifold  of similarly structured singular matrices.  While an unstructured analysis states both the same sensitivity to perturbations for each eigenvalue and the same distance to singularity of $T$ if any of its eigenvalues is the closest one to the origin, on the contrary a structured analysis reveals that the more extreme the eigenvalue closest to the origin, the more unsafe the situation, in line with the relevant structured eigenvalue sensitivity. Future work concerns the analysis of the structured distance to singularity of  symmetric tridiagonal Toeplitz-type matrices that are obtained by suitably modifying the first and last diagonal entries of a symmetric tridiagonal Toeplitz matrix, which have eigenpairs known in closed form \cite{Lo} and share with $T$ the property of having eigenvectors that do not depend on the matrix entries; see \cite[Proposition 4.15]{NR19}.

\section*{Acknowledgment}
The author is indebted  to the anonymous reviewers whose valuable comments and suggestions led to essential improvements in the analysis and presentation.

\end{document}